\documentclass[reqno]{amsart}   

\usepackage{amssymb}

\usepackage{amsthm}
\theoremstyle{plain}
\usepackage{tikz}
\usepackage{enumerate,enumitem}
\usepackage{thmtools, thm-restate}
\declaretheorem{theorem}
\usepackage{hyperref}
\usepackage{subcaption}

\newcommand{\parentheses}[1]{{\left( {#1} \right)}}

\newcommand{\p}{\parentheses}

\newcommand{\Set}[1]{{\left\lbrace {#1} \right\rbrace}}
\newcommand{\singleton}{\Set}

\def\set#1:#2{\Set{{#1} \colon {#2}}}

\numberwithin{theorem}{section}
\newtheorem{lemma}[theorem]{Lemma}

\newtheorem{exam}[theorem]{Example}

\newtheorem*{claim*}{Claim}
\newtheorem{ques}[theorem]{Question}

\newtheorem{defn}[theorem]{Definition}

\newtheorem{conj}[theorem]{Conjecture}

\newtheorem{step}{Step}

\newcommand{\N}{\mathbb{N}}

\begin{document}
\title{Hamilton cycles in infinite cubic graphs}   

\author{Max F.\ Pitz}
\address{Hamburg University, Department of Mathematics, Bundesstra\ss e 55 (Geomatikum), 20146 Hamburg, Germany}
\email{max.pitz@uni-hamburg.de}

\keywords{uniquely Hamiltonian, infinite Hamilton cycle, cubic graph}

\subjclass[2010]{05C45, 05C07}     

\begin{abstract}
Investigating a problem of B.~Mohar, we show that every one-ended Hamiltonian cubic graph with end degree $3$ contains a second Hamilton cycle. We also construct two examples showing that this result does not extend to give a third Hamilton cycle, nor that it extends to the two-ended case. 
\end{abstract}

\maketitle
\section{Overview}

In this note we investigate whether results about the Hamiltonicity of finite cubic graphs extend to the infinite setting. The term `graph' in this paper is reserved for simple graphs; when allowing parallel edges or loops, we explicitly use the term `multi-graph'. Our terminology follows \cite{Diestel}.

\subsection{Hamiltonicity in finite regular graphs}
The starting point of this paper are the following results and conjectures for finite regular graphs. 

\begin{theorem}[Smith '46, see \cite{tutte}]
\label{thm_finitecubic}
Every Hamiltonian finite cubic graph has at least two Hamilton cycles.
\end{theorem}

Here, a graph is \emph{Hamiltonian} if it contains a Hamilton cycle. A graph with precisely one Hamilton cycle is also called \emph{uniquely Hamiltonian}. Sheehan conjectured that finite cycles are the only examples of uniquely Hamiltonian regular graphs. 

\begin{conj}[Sheehan '75, \cite{sheehan}]
Every $d$-regular Hamiltonian finite graph with $d \geq 3$ has at least two Hamilton cycles.
\end{conj}

For more details on Sheehan's conjecture, we refer the reader to \cite{Thomassen}.

Using a nice parity argument, the so-called ``lollypop technique'', Thomason extended Smith's result in a different direction as follows:

\begin{theorem}[Thomason '78, \cite{thomason}]
\label{thm_finitecubicthomason}
Every edge in a finite graph with odd degrees only lies on an even number of Hamilton cycles. Hence, every Hamiltonian such graph has at least three Hamilton cycles.
\end{theorem}

In particular, every finite Hamiltonian cubic graph contains at least three Hamilton cycles.

\subsection{Infinite Hamilton circles}
For a locally finite graph $G$, which can be considered as a topological space using the $1$-complex topology, we let $|G|$ denote its Freudenthal compactification. Extending the notion of cycles, one defines \emph{circles} in $|G|$ as homeomorphic images of the unit circle $|G|$, see \cite[\S8]{Diestel}. A circle of $|G|$ is a \emph{Hamilton circle}, if it contains all vertices (and all ends) of $G$. We use the term \emph{Hamilton cycle} to denote the subgraph induced by a Hamilton circle of $|G|$.

In one-ended graphs, Hamilton cycles correspond to spanning double rays. In a two-ended graph $G$, a Hamilton cycle consists of two vertex-disjoint double rays $R_1$ and $R_2$ which together span $G$, such that the two tails of each $R_i$ belong to different ends of $G$. For example, the 2-ended double ladder in Figure~\ref{infdoubleladder} has a unique Hamilton cycle comprised of all horizontal edges.

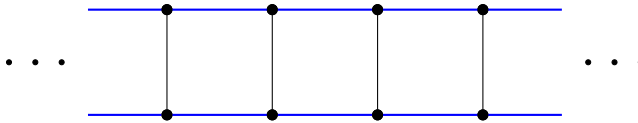
\begin{figure}[ht!]
\centering
\begin{tikzpicture}[scale=.7]

\foreach \s in {1,...,4}
{
\node[draw, circle,scale=.4, fill] (blobu\s) at (2*\s,1) {};
\node[draw, circle,scale=.4, fill] (blobl\s) at (2*\s,-1) {};
}

\foreach \s in {1,...,3}
{
\pgfmathsetmacro\t{\s+1}
\draw[blue, thick] (blobu\s) -- (blobu\t);
\draw[blue, thick] (blobl\s) -- (blobl\t);
\draw (blobu\s) -- (blobl\s);
}
\draw (blobu4) -- (blobl4);

\draw[blue, thick] (blobu4) -- (9.5,1);
\draw[blue, thick] (blobl4) -- (9.5,-1);

\draw[blue, thick](blobu1) -- (0.5,1);
\draw[blue, thick] (blobl1) -- (0.5,-1);

\node[draw, circle,scale=.2, fill] (blobm1) at (10,0) {};
\node[draw, circle,scale=.2, fill] (blobm1) at (10.5,0) {};
\node[draw, circle,scale=.2, fill] (blobm1) at (11,0) {};

\node[draw, circle,scale=.2, fill] (blobm1) at (0,0) {};
\node[draw, circle,scale=.2, fill] (blobm1) at (-0.5,0) {};
\node[draw, circle,scale=.2, fill] (blobm1) at (-1,0) {};

\foreach \s in {1,...,4}
{
\node[draw, circle,scale=.4, fill] (blobu\s) at (2*\s,1) {};
\node[draw, circle,scale=.4, fill] (blobl\s) at (2*\s,-1) {};
}

\end{tikzpicture}
\caption{The infinite double ladder and its unique Hamilton cycle.}
\label{infdoubleladder}
\end{figure}

\subsection{Questions on Hamiltonicity in infinite regular graphs}

In 2007, Mohar asked to what extent the above results about Hamiltonicity in finite regular graphs generalise to the infinite setting. While the infinite double ladder in Figure~\ref{infdoubleladder} witnesses that Theorem~\ref{thm_finitecubic} fails to extend verbatim to the infinite case, Mohar suggested two possible solutions. 

First, we might restrict out attention to one-ended graphs, and second, we might say that the double ladder is not truly regular, as its ends have degree $2$. Here, we take the \emph{degree of an end} to be the maximum number of edge-disjoint rays leading to that end, see \cite{euler} or Section~\ref{sec_2} below for details.

\begin{ques}[Mohar '07, \cite{mohar}]
\label{ques_1}
Does there exist a uniquely Hamiltonian, one-ended, $d$-regular graph for $d \geq 3$?
\end{ques}

\begin{ques}[Mohar '07, \cite{mohar}]
\label{ques_2}
Does there exist a uniquely Hamiltonian, $d$-regular graph for $d \geq 3$ where also all ends have degree $d$?
\end{ques}

K.~Heuer  \cite{karl} has recently constructed a uniquely Hamiltonian cubic graph with continuum many ends where all ends have degree 3, thus answering Question~\ref{ques_2}. He left open the natural question whether simultaneously restricting the number of ends plus the end-degrees allows us to extend finite theorems to the infinite setting.

\subsection{Results}

In this note, we establish the following extension of Smith's Theorem~\ref{thm_finitecubic} about second Hamilton cycles to the infinite setting, providing a partial answer to Mohar's questions.

\begin{theorem}
\label{thm_mainoverview}
Every Hamiltonian one-ended cubic graph with end degree at most $3$ has at least two Hamilton cycles.
\end{theorem}

Our proof of Theorem~\ref{thm_mainoverview} combines the stronger of the finite results, namely Thomason's Theorem~\ref{thm_finitecubicthomason}, and a sequence of parity arguments. Interestingly, Thomason's Theorem~\ref{thm_finitecubicthomason} itself does \emph{not} extend to the above setting: we construct one-ended cubic graphs with end-degree $2$ or $3$ that have precisely two Hamilton cycles, see Examples~\ref{exam31} and \ref{exam33}.

Improving on Heuer's example, we also construct in Example~\ref{example2ended} a two-ended, uniquely Hamiltonian, cubic graph where both ends have degree 3. This shows that in general, it is only in the one-ended case where one could hope for an affirmative result about second Hamilton circles in cubic graphs. 

Finally, we remark that we do not know whether every Hamiltonian one-ended cubic graph with end-degree $4$ has a second Hamilton cycle---this seems to be the next crucial case in attacking Question~\ref{ques_1}.

\section{Two facts about end degrees}
\label{sec_2}

In our proofs below we need two facts about end degrees in locally finite graphs. Given a graph $G=(V,E)$ and a set of vertices $S \subset V$, we denote by $E(S,V \setminus S) \subset E$ the set of edges of $G$ with one endvertex in $S$ and the other in the complement of $S$. We also abbreviate $E(v) = E(\singleton{v}, V \setminus \singleton{v})$. 

Following \cite{euler}, for an end $\omega$ of some locally finite graph $G$ we take its \emph{degree} (to be precise: its \emph{edge-degree}) to be the maximum number of edge disjoint rays in $G$ leading to $\omega$, and its \emph{vertex-degree} to be the maximum number of vertex-disjoint rays in $G$ leading to $\omega$.

\begin{lemma}[{\cite[Lemma~10]{euler}}]
\label{lem_menger}
Let $\omega$ be an end of a locally finite graph $G$ and $S\subset V(G)$ a finite vertex set. Then the maximal number of edge-disjoint rays to $\omega$ starting in $S$ equals the minimum cardinality of an edge cut separating $S$ from $\omega$.
\end{lemma}

\begin{lemma}
\label{lem_degreescoincide}
In cubic graphs, edge- and vertex-degree of ends coincide. 
\end{lemma}
\begin{proof}
In any locally finite graph, the vertex-degree of a given end is at most its edge-degree. Conversely, any family $\set{R_i}:{i \in I}$ of edge disjoint rays in a cubic graph have to be internally vertex-disjoint, as otherwise there would be a vertex of degree $\geq 4$. Thus, if $R'_i$ denotes the ray $R_i$ minus its initial vertex, then $\set{R'_i}:{i \in I}$ is a family of vertex-disjoint rays of the same cardinality as our initial family.
\end{proof}

\section{Affirmative results for second Hamilton cycles}

In this section, we present our positive results about the existence of additional Hamilton cycles in one-ended cubic graph with end-degree $2$ or $3$.

\begin{theorem}
\label{thm_1ended2deg}
Every Hamiltonian one-ended cubic graph with end-degree 2 has at least two Hamilton cycles.
\end{theorem}

\begin{proof}
Let $C$ be a Hamilton cycle of $G$. Since the end of $G$ has degree $2$, by Lemma~\ref{lem_menger} there is a finite vertex set $S \subset V$ with $|E(S,V \setminus S)| = 2$. 

Consider the minor $\hat{G}$ of $G$ where we contract $V \setminus S$ to a single `dummy' vertex. Then $C \restriction \hat{G}$ witnesses that $\hat{G}$ is a finite Hamiltonian graph. Moreover, $\hat{G}$ is \emph{nearly-cubic}, that is all vertices of $\hat{G}$ have degree $3$, with the exception of the dummy vertex, which has degree $2$. By \cite[Theorem~1]{entringer}, every nearly cubic Hamiltonian graph has a second Hamilton cycle. By combining the two Hamilton cycles of $\hat{G}$ with $C \setminus E(\hat{G})$, we have found two distinct Hamilton cycles of $G$.
\end{proof}

For the end-degree $3$ case, we employ the following auxiliary multi-graph which encodes how Hamilton cycles choose incident edges of certain vertices of a graph. 

\begin{defn}[Hamilton incidence multi-graph]
\label{def_haminc}
Let $v$ and $w$ be distinct vertices of a  Hamiltonian graph $G$. The \emph{Hamilton incidence multi-graph $H=H(G,v,w)$ of $G$ with respect to $v$ and $w$} is the bipartite multigraph with bipartition 
$$V(H) = [E(v)]^2 \sqcup [E(w)]^2$$
where the multiplicity of an edge $pq \in E(H)$ corresponds to the number of Hamilton cycles $D$ of $G$ with $p \cup q \subset D.$  
\end{defn}

As a concrete example of a Hamilton incidence multi-graph (which we shall meet again in Section~\ref{sec_examples} below), consider the Tutte fragment $T$ (invented by Tutte in \cite{tutte}) with leaves $\ell_x$, $\ell_y$ and $\ell_z$ as depicted in Figure~\ref{TutteFragment}. 
\begin{figure}[ht!]
\centering
\begin{tikzpicture}[scale=.6]

\node[draw, circle,scale=.4, fill] (blob1) at (0,0) {};
\node[draw, circle,scale=.4, fill] (blob2) at (0,1) {};
\node[draw, circle,scale=.4, fill] (blob3) at (3,1) {};
\node[draw, circle,scale=.4, fill] (blob4) at (5,1) {};
\node[draw, circle,scale=.4, fill] (blob5) at (5,0) {};

\draw (0,-.5) node {$\ell_x$};
\draw (0.3,.5) node {$e_x$};
\draw (-.5,1) node {$x$};

\draw (5,-.5) node {$\ell_y$};
\draw (4.7,.5) node {$e_y$};
\draw (5.5,1) node {$y$};

\draw (blob1) -- (blob2);
\draw (blob2) -- (blob3);
\draw (blob3) -- (blob4);
\draw (blob4) -- (blob5);

\node[draw, circle,scale=.4, fill] (blob6) at (0,2.1) {};
\node[draw, circle,scale=.4, fill] (blob7) at (2,2.1) {};
\node[draw, circle,scale=.4, fill] (blob8) at (3,2.1) {};

\draw (blob2) -- (blob6);
\draw (blob6) -- (blob7);
\draw (blob7) -- (blob8);
\draw (blob8) -- (blob3);

\node[draw, circle,scale=.4, fill] (blob9) at (0,3.2) {};
\node[draw, circle,scale=.4, fill] (blob10) at (1.4,3.2) {};
\node[draw, circle,scale=.4, fill] (blob11) at (3.6,3.2) {};
\node[draw, circle,scale=.4, fill] (blob12) at (5,3.2) {};

\draw (3.8,3.5) node {$v$};
\draw (2.8,4.1) node {$a$};
\draw (3,2.5) node {$b$};
\draw (5.5,3.2) node {$c$};

\draw (2.9,3.3) node {$f_a$};
\draw (3.55,2.6) node {$f_b$};
\draw (4.4,2.9) node {$f_c$};

\draw (blob6) -- (blob9);
\draw (blob9) -- (blob10);
\draw (blob10) -- (blob7);
\draw (blob8) -- (blob11);
\draw (blob11) -- (blob12);
\draw (blob4) -- (blob12);

\node[draw, circle,scale=.4, fill] (blob13) at (2.5,4) {};

\draw (blob10) -- (blob13);
\draw (blob11) -- (blob13);

\node[draw, circle,scale=.4, fill] (blob14) at (0,5) {};
\node[draw, circle,scale=.4, fill] (blob15) at (2.5,5) {};
\node[draw, circle,scale=.4, fill] (blob16) at (5,5) {};

\draw (blob6) -- (blob14);
\draw (blob14) -- (blob15);
\draw (blob15) -- (blob16);
\draw (blob12) -- (blob16);
\draw (blob13) -- (blob15);

\node[draw, circle,scale=.4, fill] (blob17) at (2.5,6) {};

\draw (blob14) -- (blob17);
\draw (blob16) -- (blob17);

\node[draw, circle,scale=.4, fill] (blob18) at (2.5,7) {};

\draw (blob17) -- (blob18);

\draw (2.5,7.5) node {$\ell_z$};
\draw (2.8,6.5) node {$e_z$};
\draw (2.5,5.7) node {$z$};

\end{tikzpicture}
\caption{The Tutte fragment $T$.}
\label{TutteFragment}
\end{figure}
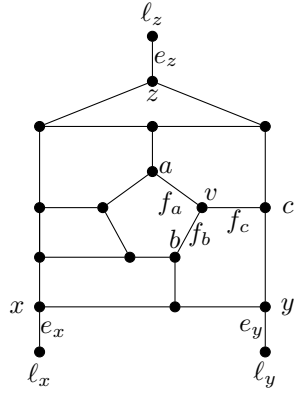

Let $T' = T / \Set{\ell_x=\ell_y=\ell_z}$ be the graph obtained from $T$ by identifying its three leaves. Then $T'$ is a cubic graph with precisely $6$ Hamilton cycles (see \cite{chia,karl,tutte}), which we may visualise as follows:

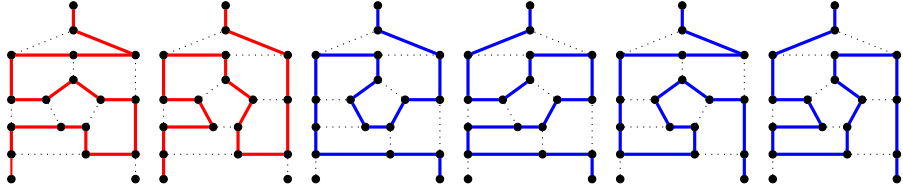
\begin{figure}[ht!]
\begin{subfigure}[t]{0.16\textwidth}
\centering
\begin{tikzpicture}[scale=.33]

\node[draw, circle,scale=.3, fill] (blob1) at (0,0) {};
\node[draw, circle,scale=.3, fill] (blob2) at (0,1) {};
\node[draw, circle,scale=.3, fill] (blob3) at (3,1) {};
\node[draw, circle,scale=.3, fill] (blob4) at (5,1) {};
\node[draw, circle,scale=.3, fill] (blob5) at (5,0) {};

\draw[red, thick] (blob1) -- (blob2);
\draw[dotted] (blob2) -- (blob3);
\draw[red, very thick] (blob3) -- (blob4);
\draw[dotted] (blob4) -- (blob5);

\node[draw, circle,scale=.3, fill] (blob6) at (0,2.1) {};
\node[draw, circle,scale=.3, fill] (blob7) at (2,2.1) {};
\node[draw, circle,scale=.3, fill] (blob8) at (3,2.1) {};

\draw[red, very thick] (blob2) -- (blob6);
\draw[red, very thick] (blob6) -- (blob7);
\draw[red, very thick] (blob7) -- (blob8);
\draw[red, very thick] (blob8) -- (blob3);

\node[draw, circle,scale=.3, fill] (blob9) at (0,3.2) {};
\node[draw, circle,scale=.3, fill] (blob10) at (1.4,3.2) {};
\node[draw, circle,scale=.3, fill] (blob11) at (3.6,3.2) {};
\node[draw, circle,scale=.3, fill] (blob12) at (5,3.2) {};

\draw[dotted] (blob6) -- (blob9);
\draw[red, very thick] (blob9) -- (blob10);
\draw[dotted] (blob10) -- (blob7);
\draw[dotted] (blob8) -- (blob11);
\draw[red, very thick] (blob11) -- (blob12);
\draw[red, very thick] (blob4) -- (blob12);

\node[draw, circle,scale=.3, fill] (blob13) at (2.5,4) {};

\draw[red, very thick] (blob10) -- (blob13);
\draw[red, very thick] (blob11) -- (blob13);

\node[draw, circle,scale=.3, fill] (blob14) at (0,5) {};
\node[draw, circle,scale=.3, fill] (blob15) at (2.5,5) {};
\node[draw, circle,scale=.3, fill] (blob16) at (5,5) {};

\draw[red, very thick] (blob9) -- (blob14);
\draw[red, very thick] (blob14) -- (blob15);
\draw[red, very thick] (blob15) -- (blob16);
\draw[dotted] (blob12) -- (blob16);
\draw[dotted] (blob13) -- (blob15);

\node[draw, circle,scale=.3, fill] (blob17) at (2.5,6) {};

\draw[dotted] (blob14) -- (blob17);
\draw[red, very thick] (blob16) -- (blob17);

\node[draw, circle,scale=.3, fill] (blob18) at (2.5,7) {};

\draw[red, very thick] (blob17) -- (blob18);

\end{tikzpicture}
\end{subfigure}%
\begin{subfigure}[t]{0.16\textwidth}
\centering
\begin{tikzpicture}[scale=.33]

\node[draw, circle,scale=.3, fill] (blob1) at (0,0) {};
\node[draw, circle,scale=.3, fill] (blob2) at (0,1) {};
\node[draw, circle,scale=.3, fill] (blob3) at (3,1) {};
\node[draw, circle,scale=.3, fill] (blob4) at (5,1) {};
\node[draw, circle,scale=.3, fill] (blob5) at (5,0) {};

\draw[red, very thick] (blob1) -- (blob2);
\draw[dotted] (blob2) -- (blob3);
\draw[red, very thick] (blob3) -- (blob4);
\draw[dotted] (blob4) -- (blob5);

\node[draw, circle,scale=.3, fill] (blob6) at (0,2.1) {};
\node[draw, circle,scale=.3, fill] (blob7) at (2,2.1) {};
\node[draw, circle,scale=.3, fill] (blob8) at (3,2.1) {};

\draw[red, very thick] (blob2) -- (blob6);
\draw[red, very thick] (blob6) -- (blob7);
\draw[dotted] (blob7) -- (blob8);
\draw[red, very thick] (blob8) -- (blob3);

\node[draw, circle,scale=.3, fill] (blob9) at (0,3.2) {};
\node[draw, circle,scale=.3, fill] (blob10) at (1.4,3.2) {};
\node[draw, circle,scale=.3, fill] (blob11) at (3.6,3.2) {};
\node[draw, circle,scale=.3, fill] (blob12) at (5,3.2) {};

\draw[dotted] (blob6) -- (blob9);
\draw[red, very thick] (blob9) -- (blob10);
\draw[red, very thick] (blob10) -- (blob7);
\draw[red, very thick] (blob8) -- (blob11);
\draw[dotted] (blob11) -- (blob12);
\draw[red, very thick] (blob4) -- (blob12);

\node[draw, circle,scale=.3, fill] (blob13) at (2.5,4) {};

\draw[dotted] (blob10) -- (blob13);
\draw[red, very thick] (blob11) -- (blob13);

\node[draw, circle,scale=.3, fill] (blob14) at (0,5) {};
\node[draw, circle,scale=.3, fill] (blob15) at (2.5,5) {};
\node[draw, circle,scale=.3, fill] (blob16) at (5,5) {};

\draw[red, very thick] (blob9) -- (blob14);
\draw[red, very thick] (blob14) -- (blob15);
\draw[dotted] (blob15) -- (blob16);
\draw[red, very thick] (blob12) -- (blob16);
\draw[red, very thick] (blob13) -- (blob15);

\node[draw, circle,scale=.3, fill] (blob17) at (2.5,6) {};

\draw[dotted] (blob14) -- (blob17);
\draw[red, very thick] (blob16) -- (blob17);

\node[draw, circle,scale=.3, fill] (blob18) at (2.5,7) {};

\draw[red, very thick] (blob17) -- (blob18);

\end{tikzpicture}
\end{subfigure}%
\begin{subfigure}[t]{0.16\textwidth}
\centering
\begin{tikzpicture}[scale=.33]

\node[draw, circle,scale=.3, fill] (blob1) at (0,0) {};
\node[draw, circle,scale=.3, fill] (blob2) at (0,1) {};
\node[draw, circle,scale=.3, fill] (blob3) at (3,1) {};
\node[draw, circle,scale=.3, fill] (blob4) at (5,1) {};
\node[draw, circle,scale=.3, fill] (blob5) at (5,0) {};

\draw[dotted] (blob1) -- (blob2);
\draw[blue, very thick] (blob2) -- (blob3);
\draw[blue, very thick] (blob3) -- (blob4);
\draw[blue, very thick] (blob4) -- (blob5);

\node[draw, circle,scale=.3, fill] (blob6) at (0,2.1) {};
\node[draw, circle,scale=.3, fill] (blob7) at (2,2.1) {};
\node[draw, circle,scale=.3, fill] (blob8) at (3,2.1) {};

\draw[blue, very thick] (blob2) -- (blob6);
\draw[dotted] (blob6) -- (blob7);
\draw[blue, very thick] (blob7) -- (blob8);
\draw[dotted] (blob8) -- (blob3);

\node[draw, circle,scale=.3, fill] (blob9) at (0,3.2) {};
\node[draw, circle,scale=.3, fill] (blob10) at (1.4,3.2) {};
\node[draw, circle,scale=.3, fill] (blob11) at (3.6,3.2) {};
\node[draw, circle,scale=.3, fill] (blob12) at (5,3.2) {};

\draw[blue, very thick] (blob6) -- (blob9);
\draw[dotted] (blob9) -- (blob10);
\draw[blue, very thick] (blob10) -- (blob7);
\draw[blue, very thick] (blob8) -- (blob11);
\draw[blue, very thick] (blob11) -- (blob12);
\draw[dotted] (blob4) -- (blob12);

\node[draw, circle,scale=.3, fill] (blob13) at (2.5,4) {};

\draw[blue, very thick] (blob10) -- (blob13);
\draw[dotted] (blob11) -- (blob13);

\node[draw, circle,scale=.3, fill] (blob14) at (0,5) {};
\node[draw, circle,scale=.3, fill] (blob15) at (2.5,5) {};
\node[draw, circle,scale=.3, fill] (blob16) at (5,5) {};

\draw[blue, very thick] (blob9) -- (blob14);
\draw[blue, very thick] (blob14) -- (blob15);
\draw[dotted] (blob15) -- (blob16);
\draw[blue, very thick] (blob12) -- (blob16);
\draw[blue, very thick] (blob13) -- (blob15);

\node[draw, circle,scale=.3, fill] (blob17) at (2.5,6) {};

\draw[dotted] (blob14) -- (blob17);
\draw[blue, very thick] (blob16) -- (blob17);

\node[draw, circle,scale=.3, fill] (blob18) at (2.5,7) {};

\draw[blue, very thick] (blob17) -- (blob18);

\end{tikzpicture}
\end{subfigure}%
\begin{subfigure}[t]{0.16\textwidth}
\centering
\begin{tikzpicture}[scale=.33]

\node[draw, circle,scale=.3, fill] (blob1) at (0,0) {};
\node[draw, circle,scale=.3, fill] (blob2) at (0,1) {};
\node[draw, circle,scale=.3, fill] (blob3) at (3,1) {};
\node[draw, circle,scale=.3, fill] (blob4) at (5,1) {};
\node[draw, circle,scale=.3, fill] (blob5) at (5,0) {};

\draw[dotted] (blob1) -- (blob2);
\draw[blue, very thick] (blob2) -- (blob3);
\draw[blue, very thick] (blob3) -- (blob4);
\draw[blue, very thick] (blob4) -- (blob5);

\node[draw, circle,scale=.3, fill] (blob6) at (0,2.1) {};
\node[draw, circle,scale=.3, fill] (blob7) at (2,2.1) {};
\node[draw, circle,scale=.3, fill] (blob8) at (3,2.1) {};

\draw[blue, very thick] (blob2) -- (blob6);
\draw[blue, very thick] (blob6) -- (blob7);
\draw[blue, very thick] (blob7) -- (blob8);
\draw[dotted] (blob8) -- (blob3);

\node[draw, circle,scale=.3, fill] (blob9) at (0,3.2) {};
\node[draw, circle,scale=.3, fill] (blob10) at (1.4,3.2) {};
\node[draw, circle,scale=.3, fill] (blob11) at (3.6,3.2) {};
\node[draw, circle,scale=.3, fill] (blob12) at (5,3.2) {};

\draw[dotted] (blob6) -- (blob9);
\draw[blue, very thick] (blob9) -- (blob10);
\draw[dotted] (blob10) -- (blob7);
\draw[blue, very thick] (blob8) -- (blob11);
\draw[blue, very thick] (blob11) -- (blob12);
\draw[dotted] (blob4) -- (blob12);

\node[draw, circle,scale=.3, fill] (blob13) at (2.5,4) {};

\draw[blue, very thick] (blob10) -- (blob13);
\draw[dotted] (blob11) -- (blob13);

\node[draw, circle,scale=.3, fill] (blob14) at (0,5) {};
\node[draw, circle,scale=.3, fill] (blob15) at (2.5,5) {};
\node[draw, circle,scale=.3, fill] (blob16) at (5,5) {};

\draw[blue, very thick] (blob9) -- (blob14);
\draw[dotted] (blob14) -- (blob15);
\draw[blue, very thick] (blob15) -- (blob16);
\draw[blue, very thick] (blob12) -- (blob16);
\draw[blue, very thick] (blob13) -- (blob15);

\node[draw, circle,scale=.3, fill] (blob17) at (2.5,6) {};

\draw[blue, very thick] (blob14) -- (blob17);
\draw[dotted] (blob16) -- (blob17);

\node[draw, circle,scale=.3, fill] (blob18) at (2.5,7) {};

\draw[blue, very thick] (blob17) -- (blob18);

\end{tikzpicture}
\end{subfigure}%
\begin{subfigure}[t]{0.16\textwidth}
\centering
\begin{tikzpicture}[scale=.33]

\node[draw, circle,scale=.3, fill] (blob1) at (0,0) {};
\node[draw, circle,scale=.3, fill] (blob2) at (0,1) {};
\node[draw, circle,scale=.3, fill] (blob3) at (3,1) {};
\node[draw, circle,scale=.3, fill] (blob4) at (5,1) {};
\node[draw, circle,scale=.3, fill] (blob5) at (5,0) {};

\draw[dotted] (blob1) -- (blob2);
\draw[blue, very thick] (blob2) -- (blob3);
\draw[dotted] (blob3) -- (blob4);
\draw[blue, very thick] (blob4) -- (blob5);

\node[draw, circle,scale=.3, fill] (blob6) at (0,2.1) {};
\node[draw, circle,scale=.3, fill] (blob7) at (2,2.1) {};
\node[draw, circle,scale=.3, fill] (blob8) at (3,2.1) {};

\draw[blue, very thick] (blob2) -- (blob6);
\draw[dotted] (blob6) -- (blob7);
\draw[blue, very thick] (blob7) -- (blob8);
\draw[blue, very thick] (blob8) -- (blob3);

\node[draw, circle,scale=.3, fill] (blob9) at (0,3.2) {};
\node[draw, circle,scale=.3, fill] (blob10) at (1.4,3.2) {};
\node[draw, circle,scale=.3, fill] (blob11) at (3.6,3.2) {};
\node[draw, circle,scale=.3, fill] (blob12) at (5,3.2) {};

\draw[blue, very thick] (blob6) -- (blob9);
\draw[dotted] (blob9) -- (blob10);
\draw[blue, very thick] (blob10) -- (blob7);
\draw[dotted] (blob8) -- (blob11);
\draw[blue, very thick] (blob11) -- (blob12);
\draw[blue, very thick] (blob4) -- (blob12);

\node[draw, circle,scale=.3, fill] (blob13) at (2.5,4) {};

\draw[blue, very thick] (blob10) -- (blob13);
\draw[blue, very thick] (blob11) -- (blob13);

\node[draw, circle,scale=.3, fill] (blob14) at (0,5) {};
\node[draw, circle,scale=.3, fill] (blob15) at (2.5,5) {};
\node[draw, circle,scale=.3, fill] (blob16) at (5,5) {};

\draw[blue, very thick] (blob9) -- (blob14);
\draw[blue, very thick] (blob14) -- (blob15);
\draw[blue, very thick] (blob15) -- (blob16);
\draw[dotted] (blob12) -- (blob16);
\draw[dotted] (blob13) -- (blob15);

\node[draw, circle,scale=.3, fill] (blob17) at (2.5,6) {};

\draw[dotted] (blob14) -- (blob17);
\draw[blue, very thick] (blob16) -- (blob17);

\node[draw, circle,scale=.3, fill] (blob18) at (2.5,7) {};

\draw[blue, very thick] (blob17) -- (blob18);

\end{tikzpicture}
\end{subfigure}%
\begin{subfigure}[t]{0.16\textwidth}
\centering
\begin{tikzpicture}[scale=.33]

\node[draw, circle,scale=.3, fill] (blob1) at (0,0) {};
\node[draw, circle,scale=.3, fill] (blob2) at (0,1) {};
\node[draw, circle,scale=.3, fill] (blob3) at (3,1) {};
\node[draw, circle,scale=.3, fill] (blob4) at (5,1) {};
\node[draw, circle,scale=.3, fill] (blob5) at (5,0) {};

\draw[dotted] (blob1) -- (blob2);
\draw[blue, very thick] (blob2) -- (blob3);
\draw[dotted] (blob3) -- (blob4);
\draw[blue, very thick] (blob4) -- (blob5);

\node[draw, circle,scale=.3, fill] (blob6) at (0,2.1) {};
\node[draw, circle,scale=.3, fill] (blob7) at (2,2.1) {};
\node[draw, circle,scale=.3, fill] (blob8) at (3,2.1) {};

\draw[blue, very thick] (blob2) -- (blob6);
\draw[blue, very thick] (blob6) -- (blob7);
\draw[dotted] (blob7) -- (blob8);
\draw[blue, very thick] (blob8) -- (blob3);

\node[draw, circle,scale=.3, fill] (blob9) at (0,3.2) {};
\node[draw, circle,scale=.3, fill] (blob10) at (1.4,3.2) {};
\node[draw, circle,scale=.3, fill] (blob11) at (3.6,3.2) {};
\node[draw, circle,scale=.3, fill] (blob12) at (5,3.2) {};

\draw[dotted] (blob6) -- (blob9);
\draw[blue, very thick] (blob9) -- (blob10);
\draw[blue, very thick] (blob10) -- (blob7);
\draw[blue, very thick] (blob8) -- (blob11);
\draw[dotted] (blob11) -- (blob12);
\draw[blue, very thick] (blob4) -- (blob12);

\node[draw, circle,scale=.3, fill] (blob13) at (2.5,4) {};

\draw[dotted] (blob10) -- (blob13);
\draw[blue, very thick] (blob11) -- (blob13);

\node[draw, circle,scale=.3, fill] (blob14) at (0,5) {};
\node[draw, circle,scale=.3, fill] (blob15) at (2.5,5) {};
\node[draw, circle,scale=.3, fill] (blob16) at (5,5) {};

\draw[blue, very thick] (blob9) -- (blob14);
\draw[dotted] (blob14) -- (blob15);
\draw[blue, very thick] (blob15) -- (blob16);
\draw[blue, very thick] (blob12) -- (blob16);
\draw[blue, very thick]  (blob13) -- (blob15);

\node[draw, circle,scale=.3, fill] (blob17) at (2.5,6) {};

\draw[blue, very thick] (blob14) -- (blob17);
\draw[dotted] (blob16) -- (blob17);

\node[draw, circle,scale=.3, fill] (blob18) at (2.5,7) {};

\draw[blue, very thick] (blob17) -- (blob18);

\end{tikzpicture}
\end{subfigure}

\caption{The six Hamilton cycles of $T'$.}
\label{TutteFragmentHamilton}
\end{figure}

The first two Hamilton cycles use the edge pair $e_x=\ell_x x$ and $e_z = \ell_z z$, and the other four Hamilton cycles use the edge pair $e_y=\ell_y y $ and $e_z$. In particular, there are no Hamilton cycles of $T'$ using the edge pair $\Set{e_x,e_y}$. Writing $w$ for the contracted vertex $\Set{\ell_x=\ell_y=\ell_z}$ in $T'$, and letting $v$ and its incident edges $f_a,f_b$ and $f_c$ be as indicated in Figure~\ref{TutteFragment}, we see that the Hamilton incidence graph $H=H(T',w,v)$ as in Definition~\ref{def_haminc} is given by the multigraph in Figure~\ref{TutteIncidence}.
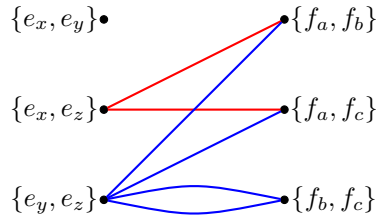
\begin{figure}[ht!]
\centering
\begin{tikzpicture}[scale=.6]

\node[draw, circle,scale=.3, fill] (blobxy) at (0,4) {};
\node[draw, circle,scale=.3, fill] (blobxz) at (0,2) {};
\node[draw, circle,scale=.3, fill] (blobyz) at (0,0) {};

\draw (-1.1,4) node {$\Set{e_x,e_y}$};
\draw (-1.1,2) node {$\Set{e_x,e_z}$};
\draw (-1.1,0) node {$\Set{e_y,e_z}$};

\node[draw, circle,scale=.3, fill] (blobab) at (4,4) {};
\node[draw, circle,scale=.3, fill] (blobac) at (4,2) {};
\node[draw, circle,scale=.3, fill] (blobbc) at (4,0) {};

\draw (5.1,4) node {$\Set{f_a,f_b}$};
\draw (5.1,2) node {$\Set{f_a,f_c}$};
\draw (5.1,0) node {$\Set{f_b,f_c}$};

\draw[red, thick] (blobxz) -- (blobab);
\draw[red, thick] (blobxz) -- (blobac);

\draw[blue, thick] (blobyz) -- (blobab);
\draw[blue, thick] (blobyz) -- (blobac);
\draw[blue, thick] (blobyz) .. controls (2,.4) .. (blobbc);
\draw[blue, thick] (blobyz) .. controls (2,-.4) .. (blobbc);

\end{tikzpicture}

\caption{The Hamilton incidence multi-graph $H(T',w,v)$.}
\label{TutteIncidence}
\end{figure}

Note that all vertices of our example $H(T',w,v)$ have even degree. In the following two lemmas, we show that this parity condition holds in general.

\begin{lemma}
\label{lem_parity}
Let $v$ and $w$ be distinct vertices of a finite cubic graph $G$. Then the sum of the degrees of any pair of vertices in the Hamilton incidence multi-graph $H(G,v,w)$ from the same side of its vertex bipartition is always even.
\end{lemma}

\begin{proof}
Indeed, if say $p \neq q \in [E(v)]^2$, we have $p \cap q = \Set{e}$ for some edge $e \in E(v)$, as $G$ is cubic. So the sum of degrees $d(p)  + d(q)$ equals the number of Hamilton cycles in $G$ using the edge $e$, which is even by Theorem~\ref{thm_finitecubicthomason}.
\end{proof}

\begin{lemma}
\label{lem_parity2}
If $v$ and $w$ are distinct vertices of a finite cubic graph $G$, then all vertex degrees in $H(G,v,w)$ are of the same parity.
\end{lemma}

\begin{proof}
Suppose one vertex in $[E(v)]^2$ has odd (even) degree. Since $|[E(v)]^2|=3$, we can apply Lemma~\ref{lem_parity} twice to conclude that all degrees on the $[E(v)]^2$ side of our bipartite graph $H=H(G,v,w)$ are odd (even). Hence,
\[
\sum_{p \in [E(v)]^2} d_H(p) = |E(H)| = \sum_{p \in [E(w)]^2} d_H(p) 
\]
is odd (even). Applying Lemma~\ref{lem_parity} twice again, we see that also all degrees on the $[E(w)]^2$ side of our bipartite graph $H$ must be odd (even). Thus, all vertex degrees in $H(G,v,w)$ are of the same parity.
\end{proof}

\begin{theorem}
\label{thm_1ended3deg}
Every Hamiltonian one-ended cubic graph with end-degree 3 has at least two Hamilton cycles.
\end{theorem}

\begin{proof}
Let $C$ be a Hamilton cycle of $G$. By assumption on the degree of our end together with Lemma~\ref{lem_menger}, there is a sequence of pairwise disjoint edge cuts $F_n = E(S_n, V \setminus S_n)$ with $S_n$ finite, $|F_n| = 3$, $S_n \subsetneq S_{n+1}$, and $\bigcup_{n \in \N}  S_n = V(G)$. 

Let $F_n = \Set{e_n,f_n,g_n}$. As every Hamilton cycle of a locally finite graph intersects each finite cut in a positive, even number of edges, \cite[8.6.7 \& 8.7.1]{Diestel}, we may suppose that $e_n,f_n \in E(C)$ and $g_n \notin E(C)$ for all $n \in \N$. Let $G_n$ be the minor of $G$ where we identify $V \setminus S_n$ to a single dummy vertex $x_n$, and let $G_{n,n+1}$ be the minor of $G$ where we identify $S_n$ and $V \setminus S_{n+1}$ to dummy vertices $v_n$ and $w_n$ respectively. 

While a priori, $G_n$ and $G_{n,n+1}$ are multi-graphs (with possibly parallel edges at dummy vertices), we may assume they are simple: By Lemma~\ref{lem_degreescoincide}, there are three vertex-disjoint rays $R_1$, $R_2$ and $R_3$ leading to the single end $\omega$. Choose $N \in \N$ such that  $E(R_i) \cap F_n \neq \emptyset$ for all $n \geq N$ and all $i$. Since the $R_i$ are vertex-disjoint, it follows that all $x_n$, $v_n$ and $w_n$ have three distinct neighbours for all $n \geq N$. 

So by moving to a suitable subsequence, we may assume that all our minors $G_n$ and $G_{n,n+1}$ are simple cubic graphs. Moreover, in all cases, the corresponding restriction of $C$ witnesses that these minors are in fact Hamiltonian. 

Now, if some $G_n$ has two distinct Hamilton cycles both using the edge set $\Set{e_n,f_n}$, then, following the same strategy as in Theorem~\ref{thm_1ended2deg}, we may combine both with $C \restriction \p{V \setminus S_n}$ to obtain two distinct Hamilton cycles of $G$. Hence, we may assume for the remainder of the proof that for all $n \in \N$, the restriction $C \restriction G_n$ is the only Hamilton cycle of $G_n$ that uses $\Set{e_n,f_n}$. In particular, we are in the case where the assumptions of the following claim are satisfied for all $n \in \N$:

 \begin{claim*}
 \label{claim_extend}
If $G_n$ and $G_{n+1}$ have unique Hamilton cycles using the edge set $\Set{e_n,f_n}$ and $\Set{e_{n+1},f_{n+1}}$ respectively, then every Hamilton cycle of $G_{n}$ extends to a Hamilton cycle of $G_{n+1}$.
 \end{claim*}

To see why the claim implies the theorem, note that by Theorem~\ref{thm_finitecubicthomason}, the edge $e_0$ is contained in an even number of Hamilton cycles of $G_0$, and hence there must be a second Hamilton cycle $A_0$ of $G_0$ which uses the edge set say $\Set{e_0,g_0}$. Applying the claim recursively, we obtain a sequence of Hamilton cycles $A_n$ of $G_n$ such that $A_{n+1}$ extends $A_n$ for all $n \in \N$. Then $A = \bigcup_{n \in \N} A_n$ a Hamilton cycle of $G$, which is distinct from $C$ witnessed by $g_0 \in E(A) \setminus E(C)$.

It remains to prove the claim. Assume that $G_n$ and $G_{n+1}$ have unique Hamilton cycles using the edge sets $\Set{e_n,f_n}$ and $\Set{e_{n+1},f_{n+1}}$ respectively, and consider the Hamilton incidence graph $H_n=H(G_{n,n+1},v_n,w_n)$ of $G_{n,n+1}$ with respect to its two dummy vertices. 

\begin{step}
\label{obs_parityone}
We have $d_{H_{n}} (\Set{e_{n+1},f_{n+1}}) = 1$.
\end{step}

This is where we use the assumption that $G_n$ and $G_{n+1}$ have unique Hamilton cycles using the edge sets $\Set{e_n,f_n}$ and $\Set{e_{n+1},f_{n+1}}$ respectively. Indeed, note first that $C \restriction G_{n,n+1}$ witnesses  that $d_{H_{n}} (\Set{e_{n+1},f_{n+1}}) \geq 1$. Next, since there is a unique Hamilton cycle $A$ of $G_n$ that uses $\Set{e_n,f_n}$, Theorem~\ref{thm_finitecubicthomason} implies that $G_n$ must have two further Hamilton cycles $B$ and $C$ using the edge sets $\Set{e_n,g_n}$ and $\Set{f_n,g_n}$ respectively. Thus, if $d_{H_{n}} (\Set{e_{n+1},f_{n+1}}) \geq 2$, i.e.\ if there are two distinct Hamilton cycles of $G_{n,n+1}$ using the edge set $\Set{e_{n+1}, f_{n+1}}$, then we can combine them suitably with either $A$, $B$ or $C$ to obtain two distinct Hamilton cycles of $G_{n+1}$ both using the edge set $\Set{e_{n+1},f_{n+1}}$, a contradiction.

\begin{step}
\label{obs_parity2}
Every vertex of $H_n$ has odd degree. 
\end{step}

Since  Step~\ref{obs_parityone} implies in particular that $d_{H_n} (\Set{e_{n+1},f_{n+1}})$ is odd, Step~\ref{obs_parity2} is immediate from Lemma~\ref{lem_parity2}.

\begin{step}
\label{obs_extending}
Every Hamilton cycle of $G_{n}$ extends to a Hamilton cycle of $G_{n+1}$.
\end{step}
Suppose we have a Hamilton cycle $A$ of $G_n$ using the edge set $p \in [F_n]^2$. By Step~\ref{obs_parity2}, we know that in particular $d_{H_n}(p) \geq 1$, which means there is a Hamilton cycle $B$ of $G_{n,n+1}$ using the edge set $p$. Then $A \cup B$ is a Hamilton cycle of $G_{n+1}$ extending $A$. This completes the proof of the final step of the claim, and so the theorem follows.
\end{proof}

\section{Examples witnessing optimality}
\label{sec_examples}

In the previous section, we have seen that Smith's Theorem~\ref{thm_finitecubic} extends to the one-ended cubic case where the end has degree at most $3$. In this section, we show that Theorem~\ref{thm_finitecubic} does not extend to the two-ended case, and that Thomason's Theorem~\ref{thm_finitecubicthomason} does not extend to the infinite case at all. 

\subsection{Ends with degree two}

\begin{exam}
\label{exam31}
There is a one-ended cubic graph with end degree $2$ that has precisely two Hamilton cycles. In particular, there are edges which do not lie on an even number of Hamilton circles.
\end{exam}

\begin{proof}[Construction]
Consider the cubic, one-way infinite ladder as in Figure~\ref{cubicdoubleladder}. Clearly, it has precisely one end, which has degree $2$.
\begin{figure}[ht!]
\centering
\begin{tikzpicture}[scale=.7]

\node[draw, circle,scale=.4, fill] (blob1) at (0.6,0) {};

\foreach \s in {1,...,4}
{
\node[draw, circle,scale=.4, fill] (blobu\s) at (2*\s,1) {};
\node[draw, circle,scale=.4, fill] (blobl\s) at (2*\s,-1) {};
}
\node[draw, circle,scale=.4, fill] (blobm1) at (2,0) {};
\draw (blob1) -- (blobu1);
\draw (blob1) -- (blobm1);
\draw (blob1) -- (blobl1);

\foreach \s in {1,...,3}
{
\pgfmathsetmacro\t{\s+1}
\draw (blobu\s) -- (blobu\t);
\draw (blobl\s) -- (blobl\t);
\draw (blobu\s) -- (blobl\s);
}
\draw (blobu4) -- (blobl4);

\draw (blobu4) -- (9.5,1);
\draw (blobl4) -- (9.5,-1);

\node[draw, circle,scale=.2, fill] (blobm1) at (10,0) {};
\node[draw, circle,scale=.2, fill] (blobm1) at (10.5,0) {};
\node[draw, circle,scale=.2, fill] (blobm1) at (11,0) {};

\draw (1.2,.8) node {$e_1$};
\draw (1.2,-.8) node {$e_2$};

\draw (2.3,.5) node {$f_1$};
\draw (2.3,-.5) node {$f_2$};

\end{tikzpicture}
\caption{The infinite cubic ladder.}
\label{cubicdoubleladder}
\end{figure}
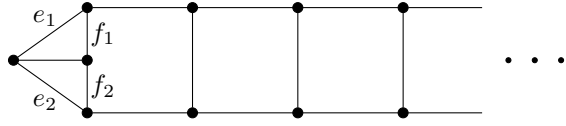
Moreover, it is not hard to check that this graph has precisely two Hamilton cycles. In particular, there are edges which do not lie on an even number of Hamilton circles. In our example, these are the edges $e_1,e_2,f_1$ and $f_2$.
\end{proof}

For completeness, we record again:

\begin{exam}
\label{example2ended}
The double ladder is a uniquely Hamiltonian, two-ended cubic graph with both ends of degree 2.
\end{exam}

\subsection{Ends with degree three}

\begin{exam}
\label{exam33}
There is a one-ended cubic graph with end degree $3$ that has precisely two Hamilton cycles. In particular, there are edges which do not lie on an even number of Hamilton circles.
\end{exam}

\begin{proof}[Construction]
Let $\set{T_n}:{n \in \N}$ be a family of disjoint graphs such that $T_0 \cong T'$ and $T_n \cong T$ for all $n \geq 1$. Here, $T$ is the Tutte fragment from Figure~\ref{TutteFragment}, and $T'$ is its cubic quotient. We use the same of vertices in $T$ and $T'$ as above, and by $v_n,a_n,b_n,c_n \in T_n$ etc.\ we refer to the respective copies of the vertices $v,a,b,c \in T$.

We now construct a sequence $\set{G_n}:{n \in \N}$ of finite graphs as follows: Put $G_0 = T_0$,
and define 
$$G_1 = \p{ G_0 - v_0 \sqcup T_1} / \sim \text{ where } a_0 \sim x_1, \;  b_0 \sim y_1, \;  c_0 \sim z_1.$$
We think of this operation as replacing the vertex $v_0$ and its incident edges by a new copy of $T$, where the leaves of the new $T$ are suitably identified with the old neighbours of $v_0$. Similarly, assuming $G_n$ has already been defined, let 
$$G_{n+1} = \p{ G_n - v_n \sqcup T_{n+1}} / \sim \text{ where } a_n \sim x_{n+1}, \;  b_n \sim y_{n+1}, \;  c_n \sim z_{n+1}.$$
In other words, in every step, we replace the most recent copy of the vertex $v$ by a new copy of $T$. 

Note that $G_n - v_n \subset G_{n+1} - v_{n+1}$ for all $n$, so we may denote by $G$ be the direct limit of these graphs. (Alternatively, $|G|$ can be viewed as  the inverse limit of the $G_n$ under natural minor relation $G_n \preceq G_{n+1}$, cf.~\cite[\S8.8]{Diestel}, and so $G$ as a 1-complex is given by the space $|G|$ minus its unique end). 

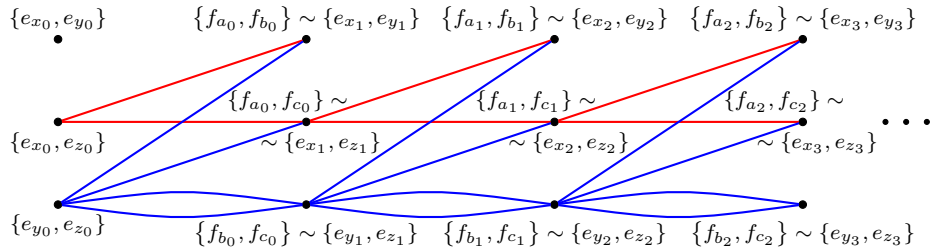
\begin{figure}[ht!]
\centering
\begin{tikzpicture}[scale=.55]

\node[draw, circle,scale=.3, fill] (blobxy) at (0,4) {};
\node[draw, circle,scale=.3, fill] (blobxz) at (0,2) {};
\node[draw, circle,scale=.3, fill] (blobyz) at (0,0) {};

\draw (0,4.5) node {\footnotesize{$\Set{e_{x_0},e_{y_0}}$}};
\draw (0,1.5) node {\footnotesize{$\Set{e_{x_0},e_{z_0}}$}};
\draw (0,-.5) node {\footnotesize{$\Set{e_{y_0},e_{z_0}}$}};

\node[draw, circle,scale=.3, fill] (blobab) at (6,4) {};
\node[draw, circle,scale=.3, fill] (blobac) at (6,2) {};
\node[draw, circle,scale=.3, fill] (blobbc) at (6,0) {};

\draw (6,4.5) node {\footnotesize{$\Set{f_{a_0},f_{b_0}} \sim \Set{e_{x_1},e_{y_1}}$}};
\draw (5.55,2.5) node {\footnotesize{$\Set{f_{a_0},f_{c_0}} \sim $}};
\draw (6.35,1.5) node {\footnotesize{$\sim \Set{e_{x_1},e_{z_1}}$}};
\draw (6,-.7) node {\footnotesize{$\Set{f_{b_0},f_{c_0}} \sim \Set{e_{y_1},e_{z_1}}$}};

\draw[red, thick] (blobxz) -- (blobab);
\draw[red, thick] (blobxz) -- (blobac);

\draw[blue, thick] (blobyz) -- (blobab);
\draw[blue, thick] (blobyz) -- (blobac);
\draw[blue, thick] (blobyz) .. controls (3,.4) .. (blobbc);
\draw[blue, thick] (blobyz) .. controls (3,-.4) .. (blobbc);

\node[draw, circle,scale=.3, fill] (blob3a) at (12,4) {};
\node[draw, circle,scale=.3, fill] (blob3b) at (12,2) {};
\node[draw, circle,scale=.3, fill] (blob3c) at (12,0) {};

\draw (12,4.5) node {\footnotesize{$\Set{f_{a_1},f_{b_1}} \sim \Set{e_{x_2},e_{y_2}}$}};
\draw (11.55,2.5) node {\footnotesize{$\Set{f_{a_1},f_{c_1}} \sim $}};
\draw (12.35,1.5) node {\footnotesize{$\sim \Set{e_{x_2},e_{z_2}}$}};
\draw (12,-.7) node {\footnotesize{$\Set{f_{b_1},f_{c_1}} \sim \Set{e_{y_2},e_{z_2}}$}};

\draw[red, thick] (blobac) -- (blob3a);
\draw[red, thick] (blobac) -- (blob3b);

\draw[blue, thick] (blobbc) -- (blob3a);
\draw[blue, thick] (blobbc) -- (blob3b);
\draw[blue, thick] (blobbc) .. controls (9,.4) .. (blob3c);
\draw[blue, thick] (blobbc) .. controls (9,-.4) .. (blob3c);

\node[draw, circle,scale=.3, fill] (blob4a) at (18,4) {};
\node[draw, circle,scale=.3, fill] (blob4b) at (18,2) {};
\node[draw, circle,scale=.3, fill] (blob4c) at (18,0) {};

\draw (18,4.5) node {\footnotesize{$\Set{f_{a_2},f_{b_2}} \sim \Set{e_{x_3},e_{y_3}}$}};
\draw (17.55,2.5) node {\footnotesize{$\Set{f_{a_2},f_{c_2}} \sim $}};
\draw (18.35,1.5) node {\footnotesize{$\sim \Set{e_{x_3},e_{z_3}}$}};
\draw (18,-.7) node {\footnotesize{$\Set{f_{b_2},f_{c_2}} \sim \Set{e_{y_3},e_{z_3}}$}};

\draw[red, thick] (blob3b) -- (blob4a);
\draw[red, thick] (blob3b) -- (blob4b);

\draw[blue, thick] (blob3c) -- (blob4a);
\draw[blue, thick] (blob3c) -- (blob4b);
\draw[blue, thick] (blob3c) .. controls (15,.4) .. (blob4c);
\draw[blue, thick] (blob3c) .. controls (15,-.4) .. (blob4c);

\node[draw, circle,scale=.2, fill] (blobm1) at (20,2) {};
\node[draw, circle,scale=.2, fill] (blobm1) at (20.5,2) {};
\node[draw, circle,scale=.2, fill] (blobm1) at (21,2) {};

\end{tikzpicture}

\caption{The incidence graph for Hamilton cycles of $G$.}
\label{TutteIncidenceG}
\end{figure}

Since $T'$ is 3-edge connected, it follows that $G$ is a one-ended cubic graph with end-degree $3$. Writing $S_n = V(G_n) \setminus \singleton{v_n}$, we see that the end-degree of $G$ is witnessed by the $3$-edge cuts 
$$F_n = E(S_n, V(G) \setminus S_n).$$ 
Moreover, if we define, as in the proof of Theorem~\ref{thm_1ended3deg}, the graphs $G_{n,n+1}$ to be the minors of $G$ where we identify $S_n$ and $V(G) \setminus S_{n+1}$ to dummy vertices $\alpha_n$ and $\beta_n$ respectively, then our construction of $G$ guarantees the existence of isomorphisms
$$\varphi_n \colon T' \to G_{n,n+1} \text{ such that }\varphi_n(w) = \alpha_n \text{ and } \varphi_n(v) = \beta_n$$
such that, due to our choice of the quotient patterns $\sim$, 
$$(\dagger) \quad \quad \varphi_n(f_a) = \varphi_{n+1}(e_x), \; \varphi_n(f_b) = \varphi_{n+1}(e_y) \text{ and }  \varphi_n(f_c) = \varphi_{n+1}(e_z)$$
for all $n \in \N$.

Next, recall that every Hamilton cycle $C$ of $G$ restricts, for any $n \in \N$, to a Hamilton cycle of $G_{n,n+1}$, and therefore looks locally like one of the six Hamilton cycles of Figure~\ref{TutteFragmentHamilton}. Pasting together the individual Hamilton incidence graphs of $G_{n,n+1}$ (cf.\ Figure~\ref{TutteIncidence}) using the identities provided in $(\dagger)$ gives the picture of Figure~\ref{TutteIncidenceG}. And since for every Hamilton cycle $C$ of $G$ we have
$$E(C \restriction G_{n,n+1}) \cap E(\beta_n) = E(C \restriction G_{n+1,n+2}) \cap E(\alpha_{n+1})$$
we see that Hamilton cycles of $G$ are in 1-1 correspondence with those rays in the multi-graph in Figure~\ref{TutteIncidenceG} that pick a single edge from each level.

To complete the construction of Example~\ref{exam33}, we now consider the graph 
$$H = \p{T \sqcup G - w_0} / \sim \text{ where } \ell_x \sim z_0, \;  \ell_y \sim y_0, \;  \ell_z \sim x_0.$$
Figure~\ref{TutteIncidenceH} shows the analogue of Figure~\ref{TutteIncidenceG} for our new graph $H$.
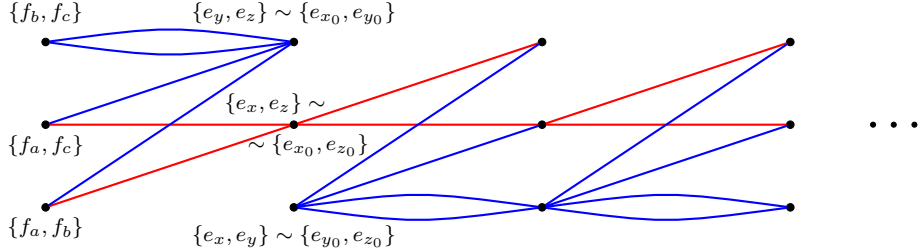
\begin{figure}[ht!]
\centering
\begin{tikzpicture}[scale=.55]

\node[draw, circle,scale=.3, fill] (blobxy) at (0,4) {};
\node[draw, circle,scale=.3, fill] (blobxz) at (0,2) {};
\node[draw, circle,scale=.3, fill] (blobyz) at (0,0) {};

\draw (0,4.7) node {\footnotesize{$\Set{f_b,f_c}$}};
\draw (0,1.5) node {\footnotesize{$\Set{f_a,f_c}$}};
\draw (0,-.5) node {\footnotesize{$\Set{f_a,f_b}$}};

\node[draw, circle,scale=.3, fill] (blobab) at (6,4) {};
\node[draw, circle,scale=.3, fill] (blobac) at (6,2) {};
\node[draw, circle,scale=.3, fill] (blobbc) at (6,0) {};

\draw (6,4.7) node {\footnotesize{$\Set{e_y,e_z} \sim \Set{e_{x_0},e_{y_0}}$}};
\draw (5.55,2.5) node {\footnotesize{$\Set{e_x,e_z} \sim $}};
\draw (6.35,1.5) node {\footnotesize{$\sim \Set{e_{x_0},e_{z_0}}$}};
\draw (6,-.7) node {\footnotesize{$\Set{e_x,e_y} \sim \Set{e_{y_0},e_{z_0}}$}};

\draw[red, thick] (blobxz) -- (blobac);
\draw[red, thick] (blobyz) -- (blobac);

\draw[blue, thick] (blobxz) -- (blobab);
\draw[blue, thick] (blobyz) -- (blobab);
\draw[blue, thick] (blobxy) .. controls (3,4.4) .. (blobab);
\draw[blue, thick] (blobxy) .. controls (3,3.6) .. (blobab);

\node[draw, circle,scale=.3, fill] (blob3a) at (12,4) {};
\node[draw, circle,scale=.3, fill] (blob3b) at (12,2) {};
\node[draw, circle,scale=.3, fill] (blob3c) at (12,0) {};

\draw[red, thick] (blobac) -- (blob3a);
\draw[red, thick] (blobac) -- (blob3b);

\draw[blue, thick] (blobbc) -- (blob3a);
\draw[blue, thick] (blobbc) -- (blob3b);
\draw[blue, thick] (blobbc) .. controls (9,.4) .. (blob3c);
\draw[blue, thick] (blobbc) .. controls (9,-.4) .. (blob3c);

\node[draw, circle,scale=.3, fill] (blob4a) at (18,4) {};
\node[draw, circle,scale=.3, fill] (blob4b) at (18,2) {};
\node[draw, circle,scale=.3, fill] (blob4c) at (18,0) {};

\draw[red, thick] (blob3b) -- (blob4a);
\draw[red, thick] (blob3b) -- (blob4b);

\draw[blue, thick] (blob3c) -- (blob4a);
\draw[blue, thick] (blob3c) -- (blob4b);
\draw[blue, thick] (blob3c) .. controls (15,.4) .. (blob4c);
\draw[blue, thick] (blob3c) .. controls (15,-.4) .. (blob4c);

\node[draw, circle,scale=.2, fill] (blobm1) at (20,2) {};
\node[draw, circle,scale=.2, fill] (blobm1) at (20.5,2) {};
\node[draw, circle,scale=.2, fill] (blobm1) at (21,2) {};

\end{tikzpicture}

\caption{The incidence graph for Hamilton cycles of $H$.}
\label{TutteIncidenceH}
\end{figure}
By the same reasoning as above, Hamilton cycles of $H$ are in 1-1 correspondence with those rays in the multi-graph in Figure~\ref{TutteIncidenceH} that pick a single edge from each level. But this means that $H$ has precisely two Hamilton cycles: Only the two left-most red edge can be extended to a ray through the multi-graph using a single edge from each level, and both these extensions are unique.
\end{proof}

\begin{exam}
\label{example2ended}
There is a uniquely Hamiltonian, two-ended cubic graph with both ends of degree 3.
\end{exam}

\begin{proof}[Construction]
For the construction, take a disjoint copy $G'$ of $G$ from the graph as constructed in the previous construction (cf. Figure~\ref{TutteIncidenceG}). By $w'_0,x'_0,y'_0,z'_0 \in G'$ etc.~we refer to the respective copies of the vertices $w_0,x_0,y_0,z_0 \in G$. Now consider the graph 
$$H' = \p{G' - w'_0 \sqcup G - w_0} \text{ with three added edges } x'_0z_0, \;  y'_0 y_0, \; \text{and} \;  z'_0x_0.$$
Then $H'$ is a $2$-ended cubic graph with both ends of degree $3$. Figure~\ref{TutteIncidenceH'} shows the analogue of Figure~\ref{TutteIncidenceH} for our new graph $H'$.

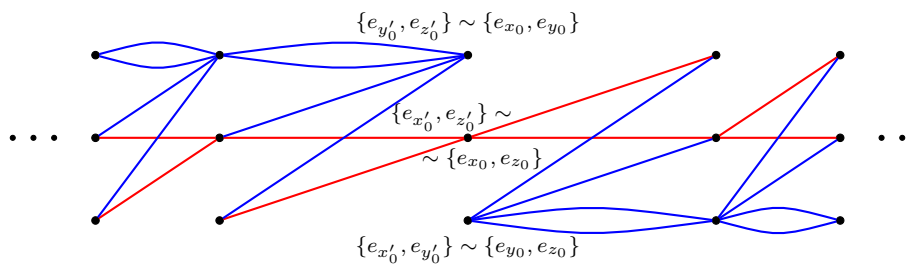
\begin{figure}[ht!]
\centering
\begin{tikzpicture}[scale=.55]

\node[draw, circle,scale=.2, fill] (blobm1) at (-4,2) {};
\node[draw, circle,scale=.2, fill] (blobm1) at (-4.5,2) {};
\node[draw, circle,scale=.2, fill] (blobm1) at (-5,2) {};

\node[draw, circle,scale=.3, fill] (blob1) at (-3,4) {};
\node[draw, circle,scale=.3, fill] (blob2) at (-3,2) {};
\node[draw, circle,scale=.3, fill] (blob3) at (-3,0) {};

\node[draw, circle,scale=.3, fill] (blobxy) at (0,4) {};
\node[draw, circle,scale=.3, fill] (blobxz) at (0,2) {};
\node[draw, circle,scale=.3, fill] (blobyz) at (0,0) {};

\draw[red, thick] (blob2) -- (blobxz);
\draw[red, thick] (blob3) -- (blobxz);

\draw[blue, thick] (blob2) -- (blobxy);
\draw[blue, thick] (blob3) -- (blobxy);
\draw[blue, thick] (blob1) .. controls (-1.5,4.4) .. (blobxy);
\draw[blue, thick] (blob1) .. controls (-1.5,3.6) .. (blobxy);

\node[draw, circle,scale=.3, fill] (blobab) at (6,4) {};
\node[draw, circle,scale=.3, fill] (blobac) at (6,2) {};
\node[draw, circle,scale=.3, fill] (blobbc) at (6,0) {};

\draw (6,4.7) node {\footnotesize{$\{e_{y'_0},e_{z'_0}\} \sim \Set{e_{x_0},e_{y_0}}$}};
\draw (5.6,2.5) node {\footnotesize{$\{e_{x'_0},e_{z'_0}\} \sim $}};
\draw (6.35,1.5) node {\footnotesize{$\sim \Set{e_{x_0},e_{z_0}}$}};
\draw (6,-.7) node {\footnotesize{$\{e_{x'_0},e_{y'_0}\} \sim \Set{e_{y_0},e_{z_0}}$}};

\draw[red, thick] (blobxz) -- (blobac);
\draw[red, thick] (blobyz) -- (blobac);

\draw[blue, thick] (blobxz) -- (blobab);
\draw[blue, thick] (blobyz) -- (blobab);
\draw[blue, thick] (blobxy) .. controls (3,4.4) .. (blobab);
\draw[blue, thick] (blobxy) .. controls (3,3.6) .. (blobab);

\node[draw, circle,scale=.3, fill] (blob3a) at (12,4) {};
\node[draw, circle,scale=.3, fill] (blob3b) at (12,2) {};
\node[draw, circle,scale=.3, fill] (blob3c) at (12,0) {};

\draw[red, thick] (blobac) -- (blob3a);
\draw[red, thick] (blobac) -- (blob3b);

\draw[blue, thick] (blobbc) -- (blob3a);
\draw[blue, thick] (blobbc) -- (blob3b);
\draw[blue, thick] (blobbc) .. controls (9,.4) .. (blob3c);
\draw[blue, thick] (blobbc) .. controls (9,-.4) .. (blob3c);

\node[draw, circle,scale=.3, fill] (blob4a) at (15,4) {};
\node[draw, circle,scale=.3, fill] (blob4b) at (15,2) {};
\node[draw, circle,scale=.3, fill] (blob4c) at (15,0) {};

\draw[red, thick] (blob3b) -- (blob4a);
\draw[red, thick] (blob3b) -- (blob4b);

\draw[blue, thick] (blob3c) -- (blob4a);
\draw[blue, thick] (blob3c) -- (blob4b);
\draw[blue, thick] (blob3c) .. controls (13.5,.4) .. (blob4c);
\draw[blue, thick] (blob3c) .. controls (13.5,-.4) .. (blob4c);

\node[draw, circle,scale=.2, fill] (blobm1) at (16,2) {};
\node[draw, circle,scale=.2, fill] (blobm1) at (16.5,2) {};
\node[draw, circle,scale=.2, fill] (blobm1) at (17,2) {};

\end{tikzpicture}

\caption{The incidence graph for Hamilton cycles of $H'$.}
\label{TutteIncidenceH'}
\end{figure}
By the same reasoning as before, Hamilton cycles of $H'$ correspond in a 1-1 fashion to those double rays in the multi-graph in Figure~\ref{TutteIncidenceH'} that pick a single edge from each level. But then it is obvious that $H'$ has a unique Hamilton cycle, which corresponds to the double ray formed by the middle horizontal edges.
\end{proof}

\end{document}